\documentclass{article}
\usepackage[utf8]{inputenc}
\usepackage{graphicx}
\usepackage{amsmath}
\usepackage[toc]{appendix}
\usepackage{amsthm}
\theoremstyle{definition} 
\theoremstyle{plain}
\newtheorem{theorem}{\newline Theorem}[section]
\theoremstyle{definition}

\theoremstyle{definition}
\newtheorem{ex}{\newline Example}[section]
\theoremstyle{definition}

\theoremstyle{definition}
\theoremstyle{definition}

\usepackage{amssymb}
\usepackage{bbm}
\usepackage{physics}
\usepackage{MnSymbol}

\usepackage{authblk}

\usepackage{tikz}
\usetikzlibrary{positioning}

\usepackage{algorithm}
\usepackage{algpseudocode}

\usepackage{color}

\usepackage{mathtools}

\usepackage{caption}
\usepackage{subcaption}

\title{Cryptanalysis of protocols using (Simultaneous) Conjugacy Search Problem in certain Metabelian Platform Groups}
\author{Delaram Kahrobaei$^{1,2,3,5}$, Carmine Monetta$^4$, Ludovic Perret$^6$, Maria Tota$^4$, Martina Vigorito$^4$}

\begin{document}

\maketitle
\begin{center}
$^1$ Department of Computer Science, University of York, UK\\
$^2$ Departments of Computer Science and Mathematics, Queens College, City
University of New York, USA\\
$^3$ Department of Computer Science and Engineering, Tandon School of Engineering,
New York University, USA\\
$^4$ Department of Mathematics, University of Salerno, IT\\
$^5$ Initiative for the Theoretical Sciences, Graduate Center, City University of New
York, USA\\
$^6$ Sorbonne University, CNRS, LIP6, PolSys, Paris, France
\end{center}

\begin{abstract}

There are many group-based cryptosystems in which the security relies on the difficulty of solving Conjugacy Search Problem (CSP) and Simultaneous Conjugacy Search Problem (SCSP) in their underlying platform groups.
In this paper we give a cryptanalysis of these systems which use certain semidirect product of abelian groups.
\end{abstract}

\section{Introduction}
The field of group-based cryptography began with the seminal work of Anshel, Anshel and Goldfeld in 1999 when they proposed a commutator key-exchange protocol based on the difficulty of simultaneus conjugacy search problem in certain classes of groups, namely braid groups \cite{AAG}. The search for the platform group for this protocol has been an active area including several cryptanalysis. For a survey on group-based cryptography in the quantum era see \cite{AMS} and book \cite{AMSB}.
Polycyclic group-based cryptography was introduced by Eick and Kahrobaei in \cite{EK}. More precisely, they proposed such groups as platform for the Commutator Key-Exchange Protocol, also known as Anshel-Anshel-Goldfeld (a.k.a. AAG) \cite{AAG}, as well as for the non-commutative Diffie-Hellman Key-Exchange Protocol (a.k.a. Ko-Lee) \cite{KoLee}. The security of these protocols relies on the difficulty to solve the Simultaneous Conjugacy Search Problem (SCSP) and the Conjgacy Search Problem (CSP) in some classes of groups. Their argument is based on experimental results for the CSP for certain metabelian polycyclic groups arising from field extensions. These groups are not virtually nilpotent,  hence the CSP cannot be solved using the analysis provided in \cite{MT}. Nevertheless, some of these groups can be avoided as platform since, in \cite{KU}, Kotov and Ushakov did a cryptanalysis for some groups of this type. A connected work is due to Gryak, Kahrobaei, and Martinez Perez who investigated another class of metabelian groups. Indeed, in \cite{GKM} they obtain a complexity result concerning the CSP which is proved to be at most exponential for the analyzed class of groups.

The methods used to test conjugacy decision problem are different and include experiments conducted with machine learning algorithms, as done by Gryak, Kahrobaei and Haralick, in \cite{GHK}, but also Length-based attack. Garber, Kahrobaei, and Lam, in \cite{GKL}, showed that the Length-based attack is inefficient for certain classes of metabelian polycyclic groups.  

There are other proposed cryptosystems based on the difficulty of CSP in certain classes of groups, (see the survey by Gryak and Kahrobaei \cite{GK}), for example Kahrobaei-Koupparis Digital Signature Scheme \cite{KK}, and Khan-Kahrobaei Non-commutative El Gamal Key-exchange \cite{KahKhan}.

In this paper we go further to the results Field-Based-Attack (FBA) in \cite{KU} and show how to cryptanalyze the CSP and SCSP for some other classes of metabelian groups. 
\\[10pt]

The authors in \cite{KU} investigated security properties of the Commutator Key-Exchange Protocol used with certain polycyclic groups. They showed that despite low success of the length based attack the protocol can be broken by a deterministic polynomial-time algorithm. They call this approach FBA and they implemented it in GAP to compare LBA and FBA.
\\[10pt]

In this paper we show that FBA could be generalized for protocols based on the difficulty of CSP and SCSP in
certain classes of metabelian groups. In particular we prove the followings theorems:\\
Theorem \ref{theoaag}: Let $G=M\ltimes N$, where $M\cong\mathbb{Z}^n$ and $N=\mathbb{Z}[m_1^{\pm1},\ldots,m_n^{\pm1}]$ (as additive groups), with $m_1,\ldots,m_n$ positive integers, then there exists a polynomial-time algorithm to break Commutator Key-Exchange protocol for such a group $G$.\\
Theorem \ref{theodh}: Let $G=M\ltimes N$, where $M\cong\mathbb{Z}^n$ and $N=\mathbb{Z}[m_1^{\pm1},\ldots,m_n^{\pm1}]$ (as additive groups), with $m_1,\ldots,m_n$ positive integers, then there exists a polynomial-time algorithm to break Diffie-Hellmann Key-Exchange protocol for such a group $G$.
\\[10pt]


This paper is structured as follow: in Section \ref{sec2}, we recall the definitions of Conjugacy Search Problem and Simultaneous Conjugacy Search Problem and we describe some Key-Exchange Protocols such as Non-commutative Diffie-Hellman and nshel-Anshel-Goldfeld Commutator. Section \ref{sec3} presents the family of metabelian groups we are interested in with some examples. In Section \ref{nondifproblems} we prove the main result i.e. how to cryptanalyze the CSP and SCSP in such platform groups and we provide the proofs of Theorem \ref{theoaag} and Theorem \ref{theodh}. The conclusions of our work are in Section \ref{conclusion}.



\section{Background}
\label{sec2}
\subsection{(Simultaneous) Conjugacy Search Problem}
We start out by giving a brief description of two group-theoretic algorithmic problems on which the security of a number of protocols is based. Here and in the following, if $x$ and $g$ are group-elements, the conjugate of $g$ by $x$, which is denoted by $g^x$, is the element $x^{-1}gx$.
\\[10pt]
\textbf{The Conjugacy Search Problem} (CSP): Let $G$ be a finitely presented group such that the conjugacy decision problem is solvable. Given $g\in G$ and $h=g^x$ for some $x\in G$, the \textit{Conjugacy Search Problem} asks to search such an element $x\in G$. 
\\[10pt]
\textbf{The Simultaneous Conjugacy Search Problem} (SCSP): Given a finitely presented group $G$ and
$g_1,\dots ,g_n$, $h_1,\dots ,h_n$ elements of $G$ such that $h_i=g_i^x$, for all $i\in \{1,\dots , n\}$ and some $x\in G$, the \textit{Simultaneous Conjugacy Search Problem} asks to recover such an element $x\in G$.
\\[10pt]
Please note that CSP and SCSP are always solvable since we assume that the decision conjugacy problem is solvable in the definitions of these problems.
Also, a solution of $g^x=h$ is not unique. In fact, given a solution $x$, the set of solutions is $\{ax\ :\ a\in C_G(g)\}$.\\
Examples of  well known protocols whose security is based on the difficulty of solving the CSP or the SCSP are the non-commutative Diffie-Hellman (a.k.a Ko-Lee) Key-Exchange Protocol and the Anshel-Anshel-Goldfeld Commutator Key-Exchange Protocol. We recall these protocols below.

\subsection{Non-commutative Diffie-Hellman (a.k.a. Ko-Lee) Key Exchange Protocol}
\label{DH}
Originally proposed by Ko, Lee, et al. \cite{KoLee} using braid groups, their non-commutative analogue of Diffie-Hellman key exchange can be generalized to work over other platform groups. Let $G$ be a finitely presented group, with $A,B \leq G$ such that all elements of $A$ and $B$ commute. 
\

An element $g\in G$ is chosen, and $g, G, A, B$ are made public. A shared secret can then be constructed as follows:
\begin{itemize}
\item Alice chooses a random element $a\in A$ and sends $g^{a}$ to Bob.
\item Bob chooses a random element $b\in B$ and sends $g^{b}$ to Alice.
\item The shared key is then $g^{ab}$, as Alice computes $(g^b)^a$, which is equal to Bob's computation of $(g^a)^b$ as $a$ and $b$ commute.
\end{itemize}

The security of such a protocol is based on the difficulty to get $a$ and $b$, which are private, from public information $g,g^a$ and $g^b$. That is to solve the conjugacy equations 
$$g^x=h\ \ \ \ \ \text{and}\ \ \ \ \ g^y=h'$$
where $h=g^a$ and $h'=g^b$.
In other words, the security of Ko-Lee rests upon solving the conjugacy search problem within the subgroups $A, B$.

\subsection{Anshel-Anshel-Goldfeld Commutator (a.k.a. AAG) Key-Exchange Protocol} 
\label{arit}
The Anshel-Anshel-Goldfeld Commutator Key-Exchange Protocol \cite{AAG} is a two-party protocol performed as follows:
\begin{itemize}
    \item Fix a finitely presented group $G$, called the platform group, a set of generators $g_1,\dots , g_k$ for $G$ and some positive integers $n_1, n_2, l, m$. All this information are made public.
    \item Alice prepares a tuple of elements $\bar{a} = (a_1,\dots, a_{n_1})$ called Alice’s public tuple. Each $a_i$ is generated randomly as a product of $g_i$’s and their inverses.
    \item Bob prepares a tuple of elements $\bar{b} = (b_1,\dots, b_{n_2})$ called Bob’s public tuple. Each $b_i$ is generated randomly as a product of $g_i$’s and their inverses.
    \item Alice generates a random element $A$ as a product $a_{s_1}^{\epsilon_1}\cdots a_{s_l}^{\epsilon_l}$ of $a_i$'s and their inverses. The element $A$ (or more precisely its factorization) is called the Alice’s private element.
    \item Bob generates a random element $B$ as a product $b_{t_1}^{\delta_1}\cdots b_{t_m}^{\delta_m}$ of $b_i$'s and their inverses. The element $B$ (or more precisely its factorization) is called the Bob’s private element.
    \item Alice publishes the tuple of conjugates $\bar{b}^A=(A^{-1}b_1A,\dots,A^{-1}b_{n_2}A)$.
    \item Bob publishes the tuple of conjugates $\bar{a}^B=(B^{-1}a_1B,\dots,B^{-1}a_{n_1}B)$.
    \item Finally, Alice computes the element $K_A$ as a product:
\[
A^{-1}(B^{-1}a_{s_1}^{\epsilon_1}B\cdots B^{-1}a_{s_l}^{\epsilon_l}B)=A^{-1}B^{-1}AB=[A,B]
\]
using the elements of Bob’s conjugate tuple $\bar{a}^B$.
\item Similarly, Bob computes the element $K_B$ as a product:
\[
(A^{-1}b_{t_1}^{\delta_1}A\cdots A^{-1}b_{t_m}^{\delta_m}A)^{-1}B=A^{-1}B^{-1}AB=[A,B]
\]
using the elements of Alice’s conjugate tuple $\bar{b}^A$.
\item The shared key is then $K=K_A=K_B=[A,B]$.
\end{itemize}
The security of such a protocol is based on the fact that it is difficult to recover $A$ and $B$ from $\bar{a},\bar{b},\bar{b}^A$ and $\bar{a}^B$, which are public.
In practice, if $\bar{b}^A=(b_1',\ ,b_{n_2}')$ and $\bar{a}^B=(a_1',\dots ,a_{n_1}')$, it is achieved by solving a system of conjugacy equations for $A$ and $B$:
\begin{equation}
\label{sistema1}
\Bigg \{
\begin{array}{ll}
X^{-1} b_1 X=b_1' \\
\dots \\
X^{-1} b_{n_2} X=b_{n_2}' \\
\end{array}
\end{equation}
\begin{equation}
\label{sistema2}
\Bigg \{
\begin{array}{ll}
Y^{-1} a_1 Y=a_1' \\
\dots \\
Y^{-1} a_{n_1} Y=a_{n_1}' \\
\end{array}
\end{equation}
This means that the security of AAG rests upon solving the simultaneous conjugacy search problem in $G$.

\section{Examples of Metabelian Groups}
\label{sec3}
Here we describe some families of metabelian groups whose CSP and SCSP will be discussed in the next section. To be more precise, we are interested in groups $G$ of the form $G=M\ltimes N$, with both groups $M$ and $N$ abelian. We use multiplicative notation for the whole group $G$ but additive notation for $N$. So if $s\in M$ and $c\in N$, the action of the element $s$ maps $c$ to
$$ c \cdot s \text{ with additive notation or,}$$
$$c^s=s^{-1}cs\text{ with multiplicative notation.}$$
This kind of groups are metabelian and arise quite naturally in linear algebra and ring theory, as we will show in more details in the following examples.

\begin{ex}\label{KoUk}
In \cite{KU}, Kotov and Ushakov studied the security of AAG protocol for some polycyclic platform groups. More precisely they considered the group $M$ as the multiplicative group of a specific field $F$ and the group $N$ as the additive group of the same field $F$; hence $G=F^*\ltimes F$. To construct $F$ they considered an irreducible monic polynomial $f(x)\in \mathbb{Z}[x]$ and put:
\begin{equation}
\label{specific field}
F=\mathbb{Q}[x]/(f).
\end{equation}
If $a\in F^*$ and $b\in F$, the action of $a$ maps $b$ to $b\cdot a$. They showed that in such a group it is possible to reduce the systems (\ref{sistema1}) and (\ref{sistema2}) to two systems of linear equations over the field $F$. Then there exist conditions under which each system has a unique solution.
\end{ex}

\begin{ex} 
\label{vectorspace}
Let $V(+,\cdot)$ be a vector space over a field $F$. Take the group $M$ as the multiplicative group $F^*$ of $F$ and the group $N$ as the additive group of $V$. If $\lambda\in F^*$ and $v\in V$, the action of $\lambda$ maps $v$ to $v\cdot \lambda$. Hence $G=F^*\ltimes V$ has the same structure of the general group we considered before. 
Notice that, for $V=F$, if $F$ is of the form described in (\ref{specific field}) we obtain the same example we found in \cite{KU}. 
Similarly we could start with a module over a commutative unitary ring.     
\end{ex}

Such examples are interesting from a mathematical point of view but more practical examples, as they have been described in \cite{GKM},  follow.

\begin{ex} \label{split} Split metabelian groups of finite Pr\"ufer rank. We will focus in the case when the group $G$ is given by a presentation of the form
$$G=\langle q_1,\ldots,q_n,b_1,\ldots,b_s\mid [q_i,q_j]=1,[b_l,b_t]=1,b_i^{q_l}=q_lb_iq_l^{-1}=b_1^{m_{1i}}b_2^{m_{2i}}\ldots b_s^{m_{si}}\rangle.$$
Observe that $q_1,\ldots,q_n$ generate a free abelian group which we denote by $M$ and  $b_1,\ldots,b_s$ generate the abelian group $N$ as normal subgroup of $G$. Then $G=M\ltimes N$. Under these conditions one can show that there is an embedding $N\to\mathbb{Q}^s$ mapping the family $b_1,\ldots,b_s$ to a free basis of $\mathbb{Q}^s$. This means that our group is torsion free metabelian of finite Pr\"ufer rank (meaning that the number of generators needed to generate any finitely generated subgroup is bounded). Observe that the action of $M$ on $N$ can be described using integer matrices: the action of $q_l$ is encoded by the $(s\times s)$-matrix $M_l$ with columns $m_{1i},\ldots,m_{s_i}$.
Moreover $G$ is polycyclic if and only if the matrices $M_i$ can be taken to be integer matrices with integral inverses \cite{AuslanderHall}.

One of the main advantages of these groups is that they admit the following fairly simple set of normal forms:
$$q_1^{\alpha_1}\ldots q_n^{\alpha_n}b_1^{\beta_1}\ldots b_s^{\beta_s}q_1^{\gamma_1}\ldots q_n^{\gamma_n}.$$
with $\gamma_1,\ldots,\gamma_n>0$.
Moreover there is an efficient algorithm (collection) to transform any word in the generators to the corresponding normal form: given an arbitrary word in the generating system,  move all of the instances of $q_i$ with negative exponent to the left and all the instances of $q_i$ with positive exponents to the right.
\end{ex}
\label{genBS}
\begin{ex} \label{genBS} Generalized metabelian Baumslag-Solitar groups. {\rm Let $m_1,\ldots,m_n$ be positive integers. We call the group given by the following presentation a {\sl generalized metabelian Baumslag-Solitar group} 
$$G=\langle q_1,\ldots,q_n,b\mid [q_i,q_j]=1, b^{q_i}=b^{m_i},i,j=1,\ldots,n\rangle.$$
It is a constructible metabelian group of finite Pr\"ufer rank and $G\cong M\ltimes N$ with $M=\langle q_1,\ldots,q_n\rangle\cong\mathbb{Z}^n$ and $N=\mathbb{Z}[m_1^{\pm1},\ldots,m_n^{\pm1}]$ (as additive groups).
} In \cite{GKM}, the authors showed the CSP in such groups reduce to the Discrete Logarithm Problem.
\end{ex}

\begin{ex}\label{galois}{\rm Let $L:\mathbb{Q}$ be a Galois extension of degree $n$ and fix an integer basis $\{u_1,\ldots,u_k\}$ of $L$ over $\mathbb{Q}$.
Then $\{u_1,\ldots,u_k\}$ freely generates the maximal order $\mathcal{O}_L$ as a $\mathbb{Z}$-module.

Now, we choose integer elements, $q_1,\ldots,q_n$, generating a free abelian multiplicative subgroup of $L-\{0\}$. Each $q_i$ acts on $L$ by left multiplication and using the  basis $\{u_1,\ldots,u_k\}$,  we may represent this action by means of an integer matrix $M_i$. Let $N$ be the smallest sub $\mathbb{Z}$-module of $L$ closed under multiplication with the elements $q_i$ and  $q_i^{-1}$  and such that $\mathcal{O}_L\subseteq N$, i.e.,
$$N=\mathcal{O}_L[q_1^{\pm 1},\ldots,q_n^{\pm1}].$$
We then may define $G=M\ltimes N$ where the action of $M$ on $N$ is given by multiplication by the elements $q_i$. {The generalized Baumslag-Solitar groups of the previous example are a particular case of this situation for $L=\mathbb{Q}$.} If the elements $q_i$ lie in $\mathcal{U}_L$ which is the group of units of $\mathcal{O}_L$, then the group $G$ is polycyclic.}
 \end{ex}

\section{Cryptanalysis of the Commutator and the Non-Commutative Diffie-Hellman key exchange Protocols}
\label{nondifproblems}

In this section, we show that the AAG and the Ko-Lee Key Exchange Protocols are not suitable in the case of the generalised metabelian Baumslag-Solitar groups (Example \ref{genBS}). Similar arguments can be used with minor modifications for the other examples in Section \ref{sec3}.
\\[10pt]
We begin studying the CSP and SCSP in a metabelian group of the form $G=M \ltimes N$, as described in Section \ref{sec3}. Assume that we have conjugated elements $g,h\in G$ and we want to solve the CSP for $g$, $h$, i.e., we want to find 
$x\in G$ such that 
$$g^x=h.$$

We put $g=sc$, $h=s'c'$ and $x=td$, where $s,s',t\in M$ and $c,c',d\in N$. Then
$$g^x=x^{-1}gx=d^{-1}t^{-1}sctd=d^{-1}st^{-1}ctd=s(d^{-1})^sc^td.$$
Now $g^x=h$ implies $s'=s$ and $c'=(d^{-1})^sc^td$. Since  the element $(d^{-1})^sc^td$ belongs to $N$ we can write it additively as
$$-d\cdot s +c\cdot t + d = d \cdot (1-s) + c \cdot t.$$
This means that the CSP above 
is equivalent to the problem of finding $t\in M$ and $d\in N$ such that
\begin{equation}\label{equation}d \cdot (1-s)+ c\cdot t=c',\end{equation}
where $s \in M$ and $c,c' \in N$ are given.
\\[10pt]
In particular, if we need to face the SCSP, which means to solve system \eqref{sistema1}, we can apply the reduction process described above. Then, if we put $b_i=s_ic_i$, $b_i'=s_i'c_i'$ with $s_i,s'_i\in M$ and $c_i,c_i'\in N$, for all $i\in\{1,\dots ,n_2\}$, and $X=td$ with $t\in M,\ d\in N$ we will get the following system of equations 

\begin{equation}
\label{sistema1ridotto}
\Bigg \{
\begin{array}{ll}
d \cdot (1-s_1)+ c_1 \cdot t = c_1' \\
\dots \\
d \cdot (1-s_{n_2} ) +c_{n_2} \cdot t=c_{n_2}'  \\
\end{array}
\end{equation}
where $s_i \in M$ and $c_i, c_i' \in N$ are given and we need to find $t \in M$ and $d \in N$.
\\[10pt]

Then the next results follow.\\
We start analyzing the cryptanalysis of AAG protocol in a generalized metabelian Baumslag-Solitar groups, as described in Example \ref{genBS}. 

\begin{theorem}
Let $G=M\ltimes N$, where $M\cong\mathbb{Z}^n$ and $N=\mathbb{Z}[m_1^{\pm1},\ldots,m_n^{\pm1}]$ (as additive groups), with $m_1,\ldots,m_n$ positive integers, then there exists a polynomial-time algorithm to break Commutator Key-Exchange protocol for such a group $G$.
\label{theoaag}
\end{theorem}
\begin{proof}
    In AAG protocol the attacker knows $b_1^X,b_2^X,\ldots,b_{n_2}^X$ for some $b_1,\ldots,b_{n_2}$ (which are public) and $n_2>1$. To find $X=td$, with $t \in M$ and $d \in N$, the attacker has to solve several equations as (\ref{sistema1ridotto}). Let us consider two of them

$$d \cdot (1-s)+ c\cdot t=c'$$
$$d \cdot (1-\tilde s)+\tilde c\cdot  t=\tilde c'.$$

Here $s,\tilde s, c, \tilde c,c',\tilde c'$ are known and the attacker has to find $t$ and $d$.
Recall that $c',\tilde c', c, \tilde c,d$ lie in $N$ which is a subring of $\mathbb{Q}$. If we identify $s$ and $t$ with the integer they act by, then they also lie in $N$. So the above can be seen as a system of two equations in $N$, moreover we know a priori that the system has a solution. This means that unless the second equation is a multiple of the first one, this solution is unique and the standard procedure to solve the system yields then the suitable value of $t$ and $d$ in polynomial time (see \cite[Section~2]{GKM}).
\end{proof}

 The argument in the previous proof applies also when $G$ is as described in Example \ref{vectorspace}, choosing $V=F^n$ with $n\in\mathbb N$.
\\[10pt]
Next, let us move to the non-commutative Diffie-Hellmann key exchange protocol (Section \ref{DH}). 
\begin{theorem}
    Let $G=M\ltimes N$, where $M\cong\mathbb{Z}^n$ and $N=\mathbb{Z}[m_1^{\pm1},\ldots,m_n^{\pm1}]$ (as additive groups), with $m_1,\ldots,m_n$ positive integers, then there exists a polynomial-time algorithm to break Diffie-Hellmann Key-Exchange protocol for such a group $G$.
    \label{theodh}
\end{theorem}
\begin{proof}
    In Ko-Lee protocol the main problem is that Alice and Bob must agree on a set $\Omega$ of pairwise commuting elements and then choose their conjugators $a$ and $b$ from that set. Recall that we are denoting $G=M\ltimes N$. As $M$ is abelian a possible choice would be $\Omega=M$, and if $a$ lies in $M$ then the attacker can find $a$ from $g^{a}$ in polynomial time. 
Another possibility would be to choose $a,b\in N$. But then $a=d$ and equation (\ref{equation}) for $g^{a}$ is 
$$d \cdot (1-s)+c=c',$$
and the only unknown is $d$ which can be found easily in polynomial time (see \cite[Section~3]{GKM}). 

In the case when $a$ is an arbitrary element not in $M$ or $N$, $\Omega$ must be a subset of the centralizer $C_G(a)$ of $a$ in $G$. 

 Things are particularly easy in the case when  the element $a$ belongs to $M^r$ for some $r\in N$, which happens if and only if
 $$a=td=t_1^r=r^{-1}t_1r=t_1 t_1^{-1}r^{-1}t_1r=t_1 r^{-t_1} r,$$
 for some $t,t_1 \in M$ and $d \in N$.
 Additively this is equivalent to 
 $$d=r- r\cdot t=r \cdot (1-t).$$
 It is a standard fact that $M^r=\{x\delta(x)\mid x\in M\}$ where $\delta$ is the inner derivation given by $\delta(x)=r \cdot (1-x)$. In this case it is easy to check that
 $$\Omega\subseteq C_G(a)=M^r.$$
 If the attacker has the extra information that $a$ belongs to $M^r$ for some $r$, then  the equation that he has to solve is 
 $$ r \cdot (1-t)(1-s)+c\cdot t=c'$$
 equivalently
$$(c-r+ r\cdot s)\cdot t=c'-r+s\cdot r.$$
This can be seen as an equation in $\mathbb{Q}$ and only requires to perform the quotient of $c'-r+ r\cdot s$ by $c-r+ r\cdot s$ thus can be solved in polynomial time (see \cite[Section~3]{GKM}).

Moreover, we are going to see now that by embedding our group $G$ in a bigger group we may always assume that $a$ lies in some conjugated subgroup of $M$. Let $\tilde G= M\ltimes\tilde N$ where $\tilde N=N\otimes_\mathbb{Z} \mathbb{Q}=\mathbb{Q}$. Then $a=td$ lies in $M^r$ for some $r\in\tilde N$ if and only if $d=r\cdot (1-t).$ This can always be solved in $\mathbb{Q}$, in other words, we can always find a suitable $r\in\mathbb{Q}$. Then one proceeds as we did before with this $r$. The fact that $r$ might not belong to $N$ does not create any troubles: recall that we are dealing not with the conjugacy problem but with the conjugacy {\sl search} problem, meaning that we know a priory that our equations have a solution so the procedure above yields the right values of $t,d$ even if $r$ does not belong to $N$.\\
Observe that behind what we said above is the fact that for the group $\tilde G$, the first cohomology group $\rm{H}^1(M,\tilde N)$ is zero, thus all the complements of $\tilde N$ in $\tilde G$ are conjugated.
\end{proof}

Notice that exactly the same argument in the previous proof happens for any group $G=M\ltimes N$ with $N\subseteq\mathbb{Q}^n$ for some $n$, so can be extended to the more general version of our groups (Example \ref{split}).



\section{Conclusion}
\label{conclusion}
In this paper we do cryptanalysis for the CSP and SCSP in certain metabelian groups. In particular we show the following.
\begin{enumerate}
    \item The generalized metabelian
Baumslag-Solitar groups can not be used as platform groups in commutator key-exchange protocol.
\item The generalized metabelian
Baumslag-Solitar groups can not be used as platform groups in non-commutative Diffie-Hellman protocol.
\end{enumerate}
Finally we want to point out that this cryptanalysis could be extended to the other examples in Section \ref{sec3} and to all cryptosystems based on the difficulty of CSP and SCSP.
\section*{Acknowledgement}
DK thanks the University of Salerno (Italy), where most of this paper was discussed and written. We thank Professor Conchita Martinez-Perez for fruitful discussions. MV thanks Initiative for the Theoretical Sciences at CUNY GC which hosted her Fall 2022.


\end{document}